%% file: Codeideal_Preprint.tex
\documentclass[a4paper,10pt]{article}
\usepackage[utf8]{inputenc}

\usepackage{color}

\usepackage{amssymb}
\usepackage{amsmath}
\usepackage{amsfonts}
\usepackage{amsthm}
\usepackage{url}
\usepackage{algorithm}
\usepackage{algorithmic}
\usepackage{hyperref}

\input{somedefs}

\newcounter{exampleNo}

\newtheorem{theorem}{Theorem}[section]
\newtheorem{lemma}[theorem]{Lemma}
\newtheorem{proposition}[theorem]{Proposition}
\newtheorem{corollary}[theorem]{Corollary}

\newenvironment{example}[1][Example \arabic{exampleNo}.]{\begin{trivlist}\refstepcounter{exampleNo}
\item[\hskip \labelsep {\bfseries #1}]}{\hfill $\diamondsuit$\end{trivlist}}

\title{On Binomial Ideals Associated to Linear Codes}
\author{Natalia Dück, Karl-Heinz Zimmermann}

\begin{document}
 \maketitle
\begin{abstract}
Recently, it was shown that a binary linear code can be associated to a binomial ideal given
as the sum of a toric ideal and a non-prime ideal. 
Since then two different generalizations have been provided which coincide for the binary case.
In this paper, we establish  some connections between the two approaches. 
In particular, we show that the corresponding code ideals are related by elimination.
Finally, a new heuristic decoding method for linear codes over prime fields is discussed using Gröbner bases. 
\end{abstract}

\section{Introduction}
Digital data are exposed to errors when transmitted through a noisy channel.
But since receiving correct data is indispensable in many applications, error-correcting codes are employed to tackle this problem. 
By adding redundancy to the messages, errors can be detected and corrected~\cite{macws,vlint}. 

Gröbner bases, on the other hand, are a powerful tool that has originated from commutative algebra providing a uniform
approach to grasp a wide range of problems such as solving algebraic systems of equations, testing ideal membership,
and effective computing in residue class rings modulo polynomial ideals~\cite{adams, cls}. 

In~\cite{borges} both subjects have been linked by associating a binary linear code to 
a certain binomial ideal given as the sum of a toric ideal and a non-prime ideal. 
In addition, the authors demonstrated how the computation of the minimum distance 
can be accomplished by a Gröbner basis computation. In~\cite{ mahkhz4, mahkhz3}
this approach has been extended to codes over finite prime fields, whose associated binomial ideals are called code ideals. 
In this way, several concepts from the rich theory of toric ideals can be translated into the setting of code ideals. 

Another generalization of~\cite{borges} was given in~\cite{mahkhz6,martcodeideal,mahkhzX}.
The approach in~\cite{martcodeideal} covers the general case of linear codes over arbitrary finite fields by introducing 
the so-called generalized code ideal.

This paper pursues two objectives. 
First, both approaches are linked to provide further inside into the structure of ideals associated to linear codes.
In the binary case both approaches are the same and it will be shown that in the case of a prime field
the code ideal is an elimination ideal of the generalized code ideal. 
Furthermore, it will be proved that the reduced Gröbner basis for the generalized code ideal with respect 
to the lexicographic order can be explicitly constructed from a generator matrix in standard form.

Second, a heuristic method is introduced which allows to decode linear codes using the corresponding code ideal 
instead of the generalized code ideal. One of the main reasons for introducing the generalized
code ideal is that it enables us to apply the same procedure for decoding as in the binary case, 
but at the expense of introducing considerably more variables.
Since for codes over prime fields the code ideal provides an alternative to the generalized code ideal
which requires only a $(q-1)$-fractional amount of variables it is reasonable to look for an alternative way
of decoding based on the code ideal. We will provide such a heuristic method and discuss its pros and cons.

\section{Preliminaries}
\subsection{Linear Codes}

Let $\F_q$ denote the finite field with $q$ elements. A \textit{linear code} $\Co$ of length $n$ and dimension $k$ over $\F_q$
is the image of a one-to-one linear mapping from $\F_q^k$ to $\F_q^n$. 
In other words, the code $\Co$ is a subspace
of the vector space $\F_q^n$ of dimension $k\leq n$. Such a code $\Co$ is called an $[n,k]$ code whose elements
are called \textit{codewords} and are usually written as row vectors.

A \textit{generator matrix} for an $[n,k]$ code $\Co$ is a $k\times n$ matrix $G$ over $\F_q$ whose rows form a basis of $\Co$.
A generator matrix in reduced echelon form 
$G=\left(I_k\mid M\right)$, where $I_k$ denotes the $k\times k$ identity matrix, is said to be in \textit{standard form}
and the corresponding code $\Co$ is called \textit{systematic}.

A \textit{parity check matrix} $H$ for an $[n,k]$ code $\Co$ is an $(n-k)\times n$ matrix over $\F_q$ such that 
a word $c\in\F_q^n$ belongs to $\Co$ if and only if $cH^T=\mathbf{0}$. 
It follows that the code $\Co$ equals the kernel of the matrix $H$ given as a mapping from $\F_q^n$ to $\F_q^{n-k}$~\cite{macws, vlint}.

\subsection{Gröbner Bases and Toric Ideals}

Write $\KK[\bx] = \KK[x_1,\ldots,x_n]$ for the commutative polynomial ring in $n$ indeterminates over 
an arbitrary field $\KK$ and denote the \textit{monomials\/} in $\KK[\bx]$ by 
$\bx^{u} = x_1^{u_1}x_2^{u_2}\cdots x_n^{u_n}$, where $u=(u_1,\ldots,u_n)\in\NN_0^n$.

A \textit{monomial order\/} on $\KK[\bx]$ is a relation $\succ$ on the set of monomials in $\KK[\bx]$ 
(or equivalently, on the exponent vectors in $\NN_0^n$) satisfying:
(1) $\succ$ is a total ordering, (2) the zero vector $\mathbf{0}$ is the unique minimal element, and
(3) $u\succ v$ implies $u+w\succ v+w$ for all $u,v,w\in\NN_0^n$.

Given a monomial order $\succ$, each non-zero polynomial $f\in\KK[\bx]$ has a unique \textit{leading term}, denoted by $\lt_\succ(f)$, 
which is given by the largest involved term. 

The \textit{leading ideal\/} of an ideal $I$ w.r.t.\ a monomial order $\succ$ is the monomial ideal 
generated by the leading monomials of its elements,
\begin{align}
 \lt_\succ(I)= \langle\lt_\succ(f)\mid f\in I\rangle.
\end{align}
The Gröbner basis for an ideal $I$ in $\KK[\bx]$ w.r.t.\ $\succ$ is a finite subset $\mathcal{G}$ of $I$ with the property that
the leading terms of the polynomials in $\mathcal{G}$ generate the leading ideal of $I$, i.e., 
\begin{align}
\lt_\succ(I) = \langle\lt_\succ(g)\mid g\in \mathcal{G}\rangle.
\end{align}
A monomial $\bx^\alpha\notin\lt_\succ(I)$ is called a \textit{standard monomial\/}. The set of all
standard monomials forms a basis for the $\KK$-algebra $\KK[\bx]/I$. 

If no monomial in a Gröbner basis is redundant and for any two distinct elements $g,h\in \mathcal{G}$, 
no term of $h$ is divisible by $\lt_\succ(g)$, then $\mathcal{G}$ is called \textit{reduced}. 
A reduced Gröbner basis is uniquely determined (provided that the generators are monic) and henceforth 
the reduced Gröbner basis for an ideal $I$ w.r.t.\ $\succ$ will be denoted by $\mathcal{G}_\succ(I)$.
For more information on Gröbner basics the reader should consult~\cite{adams, becker, cls}.

A \textit{binomial} in $\KK[\bx]$ is a polynomial consisting of two terms, i.e., a binomial is
of the form $c_\alpha\bx^\alpha-c_\beta\bx^\beta$, where $\alpha,\beta\in\NN_0^n$ and $c_\alpha,c_\beta\in\KK$ are non-zero. 
A binomial is \textit{pure} if the involved monomials are relatively prime.
A \textit{binomial ideal} is an ideal generated by binomials.

Let $A=\left(a_{ij}\right)$ be a non-negative integral $m\times n$ matrix and take the polynomial 
rings $\KK[\bx]=\KK[x_1,\dots,x_n]$ and $\KK[\by]=\KK[y_1,\dots,y_m]$.
Define the $\KK$-algebra homomorphism
$\varphi:\KK[\bx]\rightarrow\KK[\by]$ by $\varphi(x_i)=y_1^{a_{1i}}y_2^{a_{2i}}\cdot y_m^{a_{mi}}$,
where $a_i=(a_{1i},a_{2i},\ldots,a_{mi})^T$ denotes the $i$th column of the matrix $A$ for $1\leq i\leq n$.
The kernel of the morphism $\varphi$ is an ideal of $\KK[\bx]$, called \textit{toric ideal} 
associated to the matrix $A$, and denoted by $I_A=\ker(\varphi)$.

For any integer $u$, write $u^+=\max\{0,u\}$ and $u^-=\left(-u\right)^+$
and for any integer vector $u=(u_1,\dots,u_n)$ define the corresponding vectors $u^+$ and $u^-$ com\-po\-nent\-wise.
Clearly, the vectors $u^+$ and $u^-$ have disjoint support and thus any vector $u\in\ZZ^n$ can 
be uniquely written as $u=u^+-u^-$. 
For instance, if $u=(2,0,-3)$, then $u^+ = (2,0,0)$ and $u^- = (0,0,3)$.
In view of this notation, the toric ideal $I_A$ is generated by pure binomials~\cite{robbiano,sturmfels}, 
\begin{align}
 I_A=\left\langle \bx^{u^+}-\bx^{u^-}\mid u=u^+-u^-\in\ker_\ZZ(A)\right\rangle,\label{eq-toricgen}
\end{align}
where $\ker_\ZZ(A)$ denotes the kernel of the matrix $A$ defined as a mapping from $\ZZ^n$ to $\ZZ^m$.

\section{The Code Ideal}

Let $\Co$ be an $[n,k]$ code over a finite field $\F_p$, where $p$ is a prime number, and let $\KK$ be an arbitrary field.
In view of~\cite{mahkhz3}, define the \textit{code ideal} in $\KK[\bx]=\KK[x_1,\dots,x_n]$ associated to the code $\Co$ as a sum of binomial ideals
\begin{align}
 I(\Co)=I'(\Co)+I_p
\end{align}
where 
\begin{align}
I'(\Co) = \langle \bx^c-\bx^{c'}\mid c-c'\in\Co \rangle
\end{align}
and
\begin{align}
I_p=\left\langle x_i^p-1\mid 1\leq i\leq n\right\rangle.
\end{align}
In terms of the ideal $I_p$, the exponent of any monomial can be treated as a vector 
in $\F_p^n$ because for any $1\leq i\leq n$ and $0\leq r\leq p-1$,
\begin{align*}
 x_i^{p+r}\equiv x_i^{p+r}-x_i^r\cdot(x_i^p-1)=x_i^r\mod I_p
\end{align*}
and thus by induction for any integer $m\geq 0$,
\begin{align*}
x_i^{m\cdot p+r}\equiv x_i^r\mod I_p.
\end{align*}
The code ideal $I(\Co)$ can be based on a toric ideal.
To see this, let $H$ be a parity check matrix for the code $\Co$ 
and let $H'$ be an integral matrix such that $H=H'\otimes_\ZZ\F_p$. 
Then 
\begin{align}
 I(\Co)=I_{H'}+I_p.
\end{align}
It follows that the code ideal is the sum of a prime ideal and a non-prime ideal. 
Although the code ideal is not toric, it resembles a toric ideal in some respects. 
Similar to~(\ref{eq-toricgen})
the code ideal is generated by pure binomials $\bx^u-\bx^{u'}$, where $u-u'$ belongs to the kernel of $H$.

The binomial $\bx^u-\bx^{u'}\in I(\Co)$ is said to {\em correspond\/} to the codeword $u-u'$. 
However, note that in contrast to the integral case, there is no unique way of writing $u=u^+-u^-$. 
For example, the word $(1,1,0)$ in $\F_2^3$ can be written as $(1,1,0)=(0,1,0)-(1,0,0)$ 
or $(1,1,0)=(1,0,0)-(0,1,0)$. 
Thus, different binomials may correspond to the same codeword. 

Note that the reduced Gröbner basis w.r.t.\ the lexicographic (lex) ordering with $x_1\succ x_2\succ\ldots\succ x_n$
can be directly read off from a generator matrix in standard form~\cite{mahkhz4}.
More specifically, let $G$ be a standard generator matrix for $\Co$ with row vectors $\mathbf{g}_i=\mathbf{e}_i-\mathbf{m}_i$,
where $\mathbf{e}_i$ denotes the $i$th unit vector, $1\leq i\leq k$.
Then the reduced Gröbner basis $\mathcal{G}_\succ(I(\Co))$ w.r.t.\ the lex ordering is given by
\begin{align}
\mathcal{G}_\succ(I(\Co))= \left\{ x_i-\bx^{\mathbf{m}_i}\mid 1\leq i\leq k\right\}\cup
\left\{ x_i^p-1\mid k+1\leq i\leq n\right\}.
\label{eq-IclexGB}
\end{align}

Given an $[n,k]$ code $\Co$ over $\F_p$.
The corresponding code ideal can be considered as an elimination ideal of a 
toric ideal~\cite[Remark~1]{marmartinez}. This will be specified in Prop.~\ref{prop-toricLink}.
Beforehand, we require further definitions. 

To a non-negative integral $m\times n$ matrix $A$ associate the integral $m\times (m+n)$ matrix 
\begin{align}
A(p)=\left(A\mid p \cdot I_m\right).\label{eq-matrixmodext}
\end{align} 
For an integral matrix $A\in\ZZ^{m\times n}$, let $\ker_\ZZ(A)$ denote the kernel of $A$ as a mapping from $\ZZ^n$ to $\ZZ^m$,
and let $\ker_p(A)$ denote the kernel of the matrix $A\otimes_{\ZZ}\ZZ_p$ as a mapping from $\ZZ^n_p$ to $\ZZ_p^m$. 
Note that for any vector $u\in\ZZ_p^n$, $u\in\ker_p(A)$ is equivalent to $Au\equiv\mathbf{0}\mod  p$ or 
$Au=pv$ for some $v\in\ZZ^m$, which in turn is equivalent to $(u,-v)\in\ker_\ZZ(A(p))$ for some $v\in\ZZ^m$.
In other words, there is a bijective correspondence between $\ker_\ZZ(A(p))$ and $\ker_p(A)$
given by the projection onto the first $n$ coordinates.

In the following, the toric ideal associated to the matrix $A(p)$ is studied in the 
polynomial ring $\KK[\bx,\by] = \KK[x_1,\dots,x_n,y_1,\dots,y_m]$. 
\begin{proposition}\label{prop-toricLink}
Let $I(\Co)$ be the code ideal of an $[n,k]$ code $\Co$ over $\F_p$ with parity check matrix $H$ 
and let $I_{H'(p)}$ be the toric ideal associated to the non-negative integral matrix $H'$ with 
$H=H'\otimes_\ZZ \F_p$. The code ideal $I(\Co)$ is given as elimination ideal 
\begin{align}
 I(\Co)=\left(I_{H'(p)}+\left\langle y_{1}-1,\dots,y_{m}-1\right\rangle\right)\cap\KK[\bx].\label{eq-Icaselim}
\end{align}
Equivalently,
\begin{align}
 I(\Co)=\left\{f(\bx,\mathbf{1})\mid f\in I_{H'(p)}\right\},\label{eq-prop2}
\end{align}
where $\mathbf{1}$ denotes the all-1 vector.
\end{proposition}
\begin{proof}
It is clear that the statements (\ref{eq-Icaselim}) and (\ref{eq-prop2}) are equivalent. Thus it is
sufficient to prove that (\ref{eq-prop2}) holds. 

Let $\bx^a-\bx^b\in I(\Co)$ and so $a-b\in\ker_p(H')$.
By the preceding remark, there is a vector $d\in\ZZ^m$ such that $(a-b,d)\in\ker_\ZZ(H'(p))$ and so 
$\bx^a\by^{d^+}-\bx^b\by^{d^-}$ belongs to $I_{H'(p)}$. 
Conversely, let $\bx^a\by^{a'}-\bx^b\by^{b'}\in I_{H'(p)}$. Then $(a-b,a'-b')$ belongs to $\ker_\ZZ(H'(p))$
and by the bijective correspondence between $\ker_\ZZ(H'(p))$ and $\ker_p(H')$,
$a-b$ belongs to $\Co$ and thus $\bx^a-\bx^b\in I(\Co)$.
\end{proof}

\begin{example}\label{ex-H1}
Consider the $[3,2]$ code $\Co$ over $\F_7$ with generator matrix 
$$G=\begin{pmatrix}1&0&4\\0&1&1\end{pmatrix}$$
and corresponding parity check matrix $H=\begin{pmatrix}1&2&5\end{pmatrix}$.
Choose $H'=\begin{pmatrix}1&2&5\end{pmatrix}$ and so $H'(7)=\begin{pmatrix}1&2&5&7\end{pmatrix}$.
A computation in {\tt Singular}~\cite{DGPS} provides the reduced Gröbner basis for $I_{H'(7)}$ w.r.t.\ the lex ordering,
\begin{align*}
 \mathcal{G}=\{&x_3^7-y^5, x_2^5-x_3^2, x_2y^4-x_3^6, x_2^2y^3-x_3^5, x_2^3y^2-x_3^4,\\
		&x_2^4y-x_3^3, x_2x_3-y, x_1^2-x_2, x_1y-x_2^4,x_1x_3-x_2^3,x_1x_2^2-x_3,\}.	       
\end{align*}
Substituting $y=1$ for all these binomials yields the set
\begin{align*}
 \{&x_3^7-1, x_2^5-x_3^2, x_2-x_3^6, x_2^2-x_3^5, x_2^3-x_3^4,x_2^4-x_3^3, \\
		&x_2x_3-1, x_1^2-x_2, x_1-x_2^4,x_1x_3-x_2^3,x_1x_2^2-x_3\}.
\end{align*}
The reduced Gröbner basis for the ideal generated by these polynomials is 
\begin{align*}
\{x_1-x_3^3, x_2-x_3^6, x_3^7-1\},
\end{align*}
which coincides with the reduced Gröbner basis for $I(\Co)$ as given in~(\ref{eq-IclexGB}).
\end{example}
The matrix $H'$ in Prop.~\ref{prop-toricLink} can always be chosen to be non-negative. 
In this way, working with Laurent polynomials or an additional indeterminate can be avoided.

\section{The Generalized Code Ideal}

In the preceding section, the code ideal associated to a linear code over a finite prime field has been introduced.
Now the code ideal corresponding to a linear code over an arbitrary finite field is described following~\cite{martcodeideal}.
\medskip

For this, let $\Co$ be an $[n,k]$ code over the field $\F_q$, where $q=p^r$, $p$ is a prime, and $r\geq 1$ is an integer.
Let  $\alpha$ be a primitive element of $\F_q$, i.e., 
$\F_{q}=\left\{0,\alpha,\alpha^2,\dots,\alpha^{q-2},\alpha^{q-1}=1\right\}$. 
The {\em crossing map\/} 
$$\blacktriangle:\F_q^n\rightarrow\ZZ^{n(q-1)}$$ is defined as 
\begin{align*}
\mathbf{a}=(a_1,\dots,a_n) = (\alpha^{j_1},\dots,\alpha^{j_n})\mapsto(\mathbf{e}_{j_1},\dots,\mathbf{e}_{j_n}),
\end{align*}
where $\mathbf{e}_i$ is the $i$th unit vector of length $q-1$, $1\leq i\leq q-1$, and 
each zero coordinate is mapped to the zero vector of length $q-1$. 
The associated mapping $$\blacktriangledown:\ZZ^{n(q-1)}\rightarrow \F_q^n$$ is given as
\begin{align*}
 (j_{1,1},\dots,j_{1,q-1},j_{2,1},\dots,j_{n,q-1})\mapsto 
\left(\sum_{i=1}^{q-1}j_{1,i}\alpha^i,\dots,\sum_{i=1}^{q-1}j_{n,i}\alpha^i\right).
\end{align*}
For instance, in view of the field $\F_5=\{0,\alpha=2,\alpha^2=4,\alpha^3=3,\alpha^4=1\}$,
$$\blacktriangle(1,0,3)=\blacktriangle(\alpha^4,0,\alpha^3)
=(\mathbf{e}_4,\mathbf{0},\mathbf{e}_3)=(0,0,0,1,0,0,0,0,0,0,1,0) $$
and
$$\blacktriangledown (0,0,0,1,0,0,0,0,0,0,1,0)=(\alpha^4,0,\alpha^3).$$
Note that the mapping $\blacktriangledown$ is the left inverse of the crossing map $\blacktriangle$, 
i.e., $\blacktriangledown\circ\blacktriangle$ is the identity on $\F_q^n$, but it is not the right inverse.
\medskip

Put $\bx_j=(x_{j1},x_{j2},\dots,x_{j,q-1})$, $1\leq j\leq n$, and $\bx=(\bx_1,\dots,\bx_n)$.
Define the \textit{generalized code ideal} associated to the code $\Co$ as 
\begin{align}
 I_+(\Co)=\left\langle\bx^{\blacktriangle a}-\bx^{\blacktriangle b}\mid a-b\in\Co\right\rangle\subseteq\mathbb{K}[\bx]
\label{eq-otherIc}.
\end{align}
For instance, in view of the previous example, $x^{\blacktriangle(1,0,3)} = x^{(0,0,0,1,0,0,0,0,0,0,1,0)} = x_{14}x_{33}$.

A generating set for the code ideal $I_+(\Co)$ will contain both 
a generating set of the associated linear code and the associated scalar multiples,
and an encoding of the additive structure of the field $\F_q$~\cite{martcodeideal,mahkhzX}.
The latter can be given by the ideal $I_q$ in $\KK[\bx]$ generated by the set
\begin{eqnarray}
\bigcup_{i=1}^n \left( \left\{ x_{iu}x_{iv}-x_{iw}\mid\alpha^u+\alpha^v=\alpha^w\right\}
	  \cup \left\{ x_{iu}x_{iv}-1\mid\alpha^u+\alpha^v=0 \right\}\right). \label{eq-relationsI}
\end{eqnarray}

\begin{theorem}[\cite{martcodeideal}]\label{thm-martShapeIc}
Let $\Co$ be an $[n,k]$ code over $\F_q$ and
suppose $\mathbf{g}_1,\ldots, \mathbf{g}_k$ are the row vectors of a generator matrix for $\Co$.
The generalized code ideal associated to the code $\Co$ is 
\begin{align}\label{eq-codeidealGen}
I_+(\Co) = I_G + I_q,
\end{align}
where $I_G$ is an ideal of\/ $\KK[\bx]$ with generating set
\begin{align}
\left\{\bx^{\blacktriangle(\alpha^j \mathbf{g}_i)}-1\mid 1\leq i\leq k,1\leq j\leq q-1\right\}.
\end{align}
\end{theorem}
Note that the binomials in $I_G$ are squarefree.
The next result exhibits the type of binomials which belong to a generalized code ideal.
\begin{lemma}\label{lem-reducemon}
Let $\Co$ be an $[n,k]$ code over $\F_q$ and let $\bx^a-\bx^b$ be a binomial in $\KK[\bx]$. 
If $\blacktriangledown(a-b)\in\Co$, then $\bx^a-\bx^b\in I_+(\Co)$.
\end{lemma}
\begin{proof}
Put $a'=\blacktriangledown a$ and $b'=\blacktriangledown b$. 
Since the mapping $\blacktriangledown$ is linear, $a'-b'=\blacktriangledown a-\blacktriangledown b=\blacktriangledown(a-b)$.
Suppose $a' - b'\in \Co$.
Then by definition,
$\bx^{\blacktriangle a'}-\bx^{\blacktriangle b'}\in I_+(\Co)$. 

Claim that
$\bx^a-\bx^b\equiv \bx^{\blacktriangle a'}-\bx^{\blacktriangle b'}\mod I_q$.
Indeed, write
$$\bx^{a}=\prod_{i=1}^{n} x_{i1}^{a_{i1}}x_{i2}^{a_{i2}}\cdots x_{i,q-1}^{a_{i,q-1}}.$$
Then the $i$th entry of the word $a'$ is $a_{i1}\alpha+a_{i2}\alpha^2+\dots+a_{i,q-1}\alpha^{q-1} =\beta_i$, 
where either $\beta_i=0$ or  $\beta_i=\alpha^{\ell_i}$ for some $1\leq\ell_i\leq q-1$.
It follows that
$x_{i1}^{a_{i1}}x_{i2}^{a_{i2}}\cdots x_{i,q-1}^{a_{i,q-1}}\equiv 1\mod I_q$ or
$x_{i1}^{a_{i1}}x_{i2}^{a_{i2}}\cdots x_{i,q-1}^{a_{i,q-1}}\equiv x_{i\ell_i}\mod I_q$.
Thus $\bx^{\blacktriangle a'}=\prod_{i=1\atop\beta_i\ne 0}^{n}x_{i\ell_i}$
and so $\bx^a\equiv \bx^{\blacktriangle a'}\mod I_q$. 
Applying the same argument to $\bx^b$ establishes the claim and so the assertion.
\end{proof}

Note that the mapping $\blacktriangle$ is not linear, since e.g.\ over $\F_5$,
\begin{align*}
 \blacktriangle(1,3)+\blacktriangle(1,1)&= 
\blacktriangle(\alpha^4,\alpha^3)+\blacktriangle(\alpha^4,\alpha^4)=(\mathbf{e}_4,\mathbf{e}_3)+(\mathbf{e}_4,\mathbf{e}_4)
\end{align*}
and
\begin{align*}
\blacktriangle\left((1,3)+(1,1)\right)&=\blacktriangle(2,4)=\blacktriangle(\alpha,\alpha^2)=(\mathbf{e}_1,\mathbf{e}_2).
\end{align*}

However, the operator $\blacktriangle$ applied to the exponent of a monomial is quasi-linear as described in the following.
\begin{lemma}\label{lem-reducemonomial}
For any vectors $a,b$ in $\F_q^n$, 
 \begin{align*}
  \bx^{\blacktriangle a+\blacktriangle b}\equiv \bx^{\blacktriangle(a+b)}\mod I_q.
 \end{align*}
\end{lemma}
\begin{proof}
Let $a=\left(\alpha^{i_1},\dots,\alpha^{i_n}\right)$ and $b=\left(\alpha^{j_1},\dots,\alpha^{j_n}\right)$.
Assume that all entries in $a$ and $b$ are non-zero; the more general case can be similarly handled.
Put $a+b=\left(\alpha^{k_1},\dots,\alpha^{k_n}\right)$ and assume that the zero entries are at the positions in the set $J\subseteq\{1,\ldots,n\}$, 
i.e., $\alpha^{i_s}+\alpha^{j_s}=\alpha^{k_s}$ for $s\notin J$ and $\alpha^{i_s}+\alpha^{j_s}=0$ for $s\in J$. 
The binomials $x_{s,i_s}x_{s,j_s}-x_{s,k_s}$ for $s\notin J$ and $x_{s,i_s}x_{s,j_s}-1$ for $s\in J$ 
belong to $I_q$.
Moreover,
$$\bx^{\blacktriangle a+\blacktriangle b} = \prod_{s=1}^{n} \bx_s^{\mathbf{e}_{i_s}+\mathbf{e}_{j_s}} =\prod_{s=1}^n x_{s,i_s}x_{s,j_s}$$
and
$$\bx^{\blacktriangle(a+b)}=\prod_{s\in J\setminus\underline{n}} \bx_s^{\mathbf{e}_{k_s}} =\prod_{s\in J\setminus\underline{n}}x_{s,k_s}.$$
But $x_{s,i_s}x_{s,j_s}\equiv x_{s,k_s} \mod I_q$ for $s\in J\setminus\underline{n}$ and $x_{s,i_s}x_{s,j_s}\equiv 1 \mod I_q$ for $s\in J$.
By comparing both equations, the result follows.
\end{proof}
Note that this result has been implicitly used in~\cite[Theorem 2.1]{martcodeideal}.

\subsection{The generalized code ideal for codes over prime fields}

For a binary $[n,k]$ code $\Co$, the generalized code ideal equals the code ideal.
To see this, note that $\F_2=\{0,\alpha=1\}$ and the generalized code ideal
can be considered as an ideal in $\KK[x_1,\dots,x_n]$ instead of $\KK[x_{11},\dots,x_{n1}]$.

Moreover, if $G$ is a generator matrix for $\Co$ with rows $\mathbf{g}_1,\ldots,\mathbf{g}_k$,
then the code ideal $I(\Co)$ has the generating set~\cite[Theorem 3.2]{marmartinez}
\begin{align*}
\left\{\bx^{\mathbf{g}_i}-1\mid 1\leq i\leq k\right\}\cup\left\{x_i^2-1\mid 1\leq i\leq n\right\}. 
\end{align*}
By Thm.~\ref{thm-martShapeIc}, the generalized code ideal $I_+(\Co)$ has the same generating set.
\medskip

Now let $\Co$ be an $[n,k]$ code over a finite field $\F_p$, where $p>2$ is a prime.
Recall that $\bx=\left(\bx_1,\dots,\bx_n\right)$ and $\bx_j=\left(x_{j1},\dots,x_{j,p-1}\right)$
for $1\leq j\leq n$. 
Moreover, put $\underline{\bx}_i=(x_{1i},\dots,x_{ni})$ for $1\leq i\leq q-1$. 
The generalized code ideal belongs to the ring $\KK[\bx]=\KK[x_{11},\dots,x_{n,p-1}]$ whereas the code ideal 
\begin{align}
I(\Co)=\left\langle\underline{\bx}_i^a-\underline{\bx}_i^b\mid a-b\in\Co\right\rangle\label{eq-usualIc}
\end{align}
can be considered to belong to $\KK[\underline{\bx}_i]\subset\KK[\bx]$ for any $1\leq i\leq p-1$.
\begin{proposition}
 Let $\mathcal{C}$ be a linear code of length $n$ over a prime field $\mathbb{F}_p$. 
The code ideal $I(\mathcal{C})$ as defined in (\ref{eq-usualIc}) is an elimination ideal
of the ideal $I_+(\Co)$ as defined in (\ref{eq-otherIc}). More precisely,
$$I(\Co)=I_+(\Co)\cap\KK[\underline{\bx}_i]\quad\mbox{for each }1\leq i\leq q-1.$$
\end{proposition}
\begin{proof}
 Let $\underline{\bx}_i^a-\underline{\bx}_i^b\in I(\Co)$, i.e., $a-b\in\Co$. 
Clearly, 
$\underline{\bx}_i^a-\underline{\bx}_i^b=\bx^{a'}-\bx^{b'}$,
where $a'=(a_1 \mathbf{e}_{i},\dots,a_n \mathbf{e}_{i})$ and $b'=(b_1 \mathbf{e}_{i},\dots,b_n \mathbf{e}_{i})$. 
Furthermore,
$\blacktriangledown(a'-b')=\alpha^i (a-b)\in\Co$ and so by Lem.~\ref{lem-reducemon},
$\underline{\bx}_i^a-\underline{\bx}_i^b\in I_+(\Co)\cap\KK[\underline{\bx}_i]$.

Conversely, let $\bx^{a}-\bx^{ b}$ be a binomial in $I_+(\Co)\cap\KK[\underline{\bx}_i] $. Clearly, 
$a- b$ must be of the form 
$((a_1-b_1)\mathbf{e}_i,(a_2-b_2)\mathbf{e}_i,\dots,(a_n-b_n)\mathbf{e}_i)$ with
$$\blacktriangledown \left(a-b\right)=\alpha^i\cdot(a-b)\in\Co.$$
But as $\Co$ is linear, $\alpha^{p-i-1}(\alpha^i(a-b))=a-b\in\Co$
and so by convention, $\bx^{a}-\bx^{b}=\underline{\bx}_i^a-\underline{\bx}_i^b\in I(\Co).$
\end{proof}

\begin{example}\label{ex-1}
Consider the ternary $[6,3]$ code $\Co$ generated by the matrix
\begin{align*}
 G=\begin{pmatrix}
    1&0&0&2&2&0\\0&1&0&1&1&0\\0&0&1&1&2&1
   \end{pmatrix}.
\end{align*}
By Thm.~\ref{thm-martShapeIc}, the generalized code ideal $I_+(\Co)$ has the generators
$$\begin{array}{lll}
x_{12}x_{41}x_{51}-1, & x_{11}x_{42}x_{52}-1,       &  x_{22}x_{42}x_{52}-1, \\  
x_{21}x_{41}x_{51}-1, & x_{32}x_{42}x_{51}x_{62}-1, & x_{31}x_{41}x_{52}x_{61}-1 
\end{array}$$
and 
$$ x_{i1}^2-x_{i2}, \; x_{i1}x_{i2}-1,\; x_{i2}^2-x_{i1},  \quad 1\leq i\leq 6.$$
Computations in {\tt Singular}~\cite{DGPS} show that the elimination ideal 
$$I_+(\Co)\cap\KK[x_{11},x_{21},x_{31},x_{41},x_{51},x_{61}]$$
is generated by the binomials
\begin{align*}
x_{61}^3-1,\, x_{51}^3-1,\, x_{41}^3-1,\, x_{31}-x_{41}^2x_{51}x_{61}^2,\, x_{21}-x_{41}^2x_{51}^2,\, x_{11}-x_{41}x_{51}  .
\end{align*}
Similarly, the elimination ideal 
$$I_+(\Co)\cap\KK[x_{12},x_{22},x_{32},x_{42},x_{52},x_{62}]$$
is generated by
\begin{align*}
x_{62}^3-1,\, x_{52}^3-1,\, x_{42}^3-1,\, x_{32}-x_{42}^2x_{52}x_{62}^2,\, x_{22}-x_{42}^2x_{52}^2,\, x_{12}-x_{42}x_{52}.
\end{align*}
Comparing these generators with the reduced Gröbner basis for the code ideal $I(\Co)$ given in~(\ref{eq-IclexGB}) confirms 
that both elimination ideals are (up to renaming of variables) equal to $I(\Co)$. 
\end{example}

Next we show that the reduced Gröbner basis for a generalized code ideal w.r.t.\ the lex ordering 
can be easily constructed from a standard generator matrix.

\begin{theorem}\label{thm-lexGB}
Let $\Co$ be an $[n,k]$ code over a prime field $\F_p$ with primitive element $\alpha$ and 
let $\mathbf{g}_1,\ldots,\mathbf{g}_k$ be the row vectors of a generator matrix for $\Co$ in standard form. 
The reduced Gröbner basis for the generalized code ideal $I_+(\Co)$ w.r.t.\ the lex ordering 
$x_{11}\succ x_{12}\succ \ldots \succ x_{n,p-1}$
is given by
\begin{align}
 \mathcal{G}=
&\left\{x_{ij}-\underline{\bx}_{p-1}^{\mathbf{m}_i^{(j)}}\mid 1\leq i\leq k,\,1\leq j\leq p-1\right\}\label{eq-lexGb1}\\
&\cup\left\{x_{ij}-x_{i,p-1}^{ \alpha^j}\mid k+1\leq i\leq n,\, 1\leq j\leq p-2\right\}\label{eq-lexGb2}\\
&\cup\left\{x_{i,p-1}^p-1\mid k+1\leq i\leq n\right\}\label{eq-lexGb3}
\end{align}
where 
\begin{align}
\mathbf{m}_i^{(j)}&=\left(\mathbf{e}_i-\mathbf{g}_i\right)\alpha^j,\quad 1\leq i\leq k,\,1\leq j\leq p-1.
\end{align}
\end{theorem}
\begin{proof}
 Note that the support of each vector $\mathbf{m}_i^{(j)}$ lies in $\{k+1,\dots,n\}$.
Thus the monomial  $\underline{\bx}_{p-1}^{\mathbf{m}_i^{(j)}}$ involves only 
the variables $x_{k+1,p-1},x_{k+2,p-1},\dots,x_{n,p-1}$.
It follows that the second terms do not involve any of the leading terms. 
Moreover,  different binomials in $\mathcal{G}$ have relatively prime leadings terms.
Hence, $\mathcal{G}$ is the reduced Gröbner basis w.r.t.\ lex ordering for the ideal it generates.

It remains to show that $\mathcal{G}$ generates the ideal $I_+(\Co)$.
First, claim that $\mathcal{G}\subset I_+(\Co)$.
Indeed, consider the following cases: 
\begin{itemize}
\item 
Take a binomial $x_{ij}-\underline{\bx}_{p-1}^{\mathbf{m}_i^{(j)}}$ from the subset~(\ref{eq-lexGb1}). 
Applying the mapping $\blacktriangledown$ to the exponents of the involved monomials yields
$\alpha^j\mathbf{e}_i$ and $\alpha^{p-1}\mathbf{m}_i^{(j)}=\mathbf{m}_i^{(j)}$. 
But $\alpha^j\mathbf{e}_i-\mathbf{m}_i^{(j)}=\alpha^j\mathbf{g}_i\in\Co$
and so by Lem.~\ref{lem-reducemon} the considered binomial belongs to $I_+(\Co)$.
\item
Consider a binomial $x_{ij}-x_{i,p-1}^{\alpha^j}$ from the subset~(\ref{eq-lexGb2}).
Applying the mapping $\blacktriangledown$ to the exponents of the monomials in this binomial 
gives $\alpha^j\mathbf{e}_i-\alpha^{p-1}\alpha^j\mathbf{e}_i=\mathbf{0}\in\Co$ and thus
by Lem.~\ref{lem-reducemon} this binomial lies in $I_+(\Co)$.
\item 
Pick a binomial $x_{i,p-1}^p-1$ from the subset~(\ref{eq-lexGb3}).
It obviously corresponds to the zero word and therefore belongs to $I_+(\Co)$.
\end{itemize}
Second, claim that $I_+(\Co)\subset\left\langle\mathcal{G}\right\rangle$. 
Indeed, put
$$J = \left\langle x_{ij}-x_{i,p-1}^{ \alpha^j}\mid k+1\leq i\leq n,\, 1\leq j\leq p-2\right\rangle$$
and
$$K =\left\langle x_{i,p-1}^p-1\mid k+1\leq i\leq n\right\rangle.$$
Consider the following cases:
\begin{itemize}
\item First we prove 
that the binomials in~$I_G$ are generated by the binomials in $\mathcal{G}$:  
For this, consider the binomial $\bx^{\blacktriangle (\alpha^j\mathbf{g}_i)}-1$ 
for some $1\leq i\leq k$ and $1\leq j\leq p-1$.
By definition, $\alpha^j\mathbf{g}_i=\alpha^j\mathbf{e}_i-\mathbf{m}_{i}^{(j)}$.
Claim that
$$\bx^{\blacktriangle (\alpha^j\mathbf{g}_i)}-1 
\equiv\bx^{\blacktriangle\left(-\mathbf{m}_{i}^{(j)}\right)}
\left(x_{ij}-\underline{\bx}_{p-1}^{\mathbf{m}_{i}^{(j)}}\right)\mod J.$$
Indeed, 
$$\bx^{\blacktriangle \left(\alpha^j\mathbf{g}_i\right)}
= \bx^{\blacktriangle \left(\alpha^j\mathbf{e}_i\right)}\bx^{\blacktriangle \left(-\mathbf{m}_i^{(j)}\right)}
= x_{ij}\bx^{\blacktriangle \left(-\mathbf{m}_i^{(j)}\right)}.$$
Moreover, the squarefree monomial $\bx^{\blacktriangle\left(-\mathbf{m}_{i}^{(j)}\right)}$ has
$\supp\left(\mathbf{m}_{i}^{(j)}\right)\subseteq\{k+1,\dots,n\}$ and so only involves the variables $\bx_{k+1},\dots,\bx_{n}$.
If the variable $x_{st}$ for some $k+1\leq s\leq n$ and $1\leq t\leq p-1$ is involved  
in $\bx^{\blacktriangle\left(-\mathbf{m}_{i}^{(j)}\right)}$, 
then the $s$-th coordinate of $\mathbf{m}_i^{(j)}$, say $\alpha^m$, is non-zero 
and satisfies $-\alpha^m=\alpha^t$.
Hence, the monomial $\underline{\bx}_{p-1}^{\mathbf{m}_{i}^{(j)}}$ contains 
the variable $x_{s,p-1}^{\alpha^m}$. 

Two cases occur: If $1\leq t\leq p-2$, then $x_{st}-x_{s,p-1}^{\alpha^t}\in\mathcal{G}$ and thus
\begin{align*}
x_{st}x_{s,p-1}^{\alpha^m}\equiv x_{s,p-1}^{\alpha^t}x_{s,p-1}^{\alpha^m}
= x_{s,p-1}^{\alpha^t+\alpha^m}
\equiv x_{s,p-1}^{0}=1\mod J+K.
\end{align*}
Otherwise, $t=p-1$ and then
\begin{align*}
x_{s,p-1}x_{s,p-1}^{\alpha^m}=x_{s,p-1}^{\alpha^m+1}\equiv x_{s,p-1}^0
=1\mod J+K.
\end{align*}
Therefore, both cases provide
$\bx^{\blacktriangle\left(- \mathbf{m}_{i}^{(p-1)}\right)}\underline{\bx}_{p-1}^{\mathbf{m}_{i}^{(p-1)}}
\equiv 1\mod J+K$.

\item 
Second we prove that the binomials in~$I_q$ whose second term is unequal to 1 
are generated by the binomials in $\mathcal{G}$. 
For this, let $\alpha^u+\alpha^v=\alpha^w$ with $\alpha^u,\alpha^v,\alpha^w\neq 0$ and consider the following cases:
\begin{itemize}
 \item Let $1\leq i\leq k$. We show that $x_{iu}x_{iv}-x_{iw}\in\left\langle\mathcal{G}\right\rangle$.
Take the following polynomial which obviously belongs to $\left\langle\mathcal{G}\right\rangle$,
\begin{align*}
&\left(x_{iw}-\underline{\bx}_{p-1}^{\mathbf{m}_i^{(w)}}\right) 
-x_{iv}\left(x_{iu}-\underline{\bx}_{p-1}^{\mathbf{m}_i^{(u)}}\right)
-\underline{\bx}_{p-1}^{\mathbf{m}_i^{(u)}}\left(x_{iv}-\underline{\bx}_{p-1}^{\mathbf{m}_i^{(v)}}\right)\\
&=x_{iw}-x_{iu}x_{iv}
-\left(\underline{\bx}_{p-1}^{\mathbf{m}_i^{(w)}}
-\underline{\bx}_{p-1}^{\mathbf{m}_i^{(u)}}\underline{\bx}_{p-1}^{\mathbf{m}_i^{(v)}}\right)\\
&\equiv x_{iw}-x_{iu}x_{iv}\mod K,
\end{align*}
where the last step follows from
\begin{align}\label{eq-auxMkuerzen}
\begin{split}
\underline{\bx}_{p-1}^{\mathbf{m}_i^{(u)}}\underline{\bx}_{p-1}^{\mathbf{m}_i^{(v)}}
&=\underline{\bx}_{p-1}^{(\mathbf{e}_i-\mathbf{g}_i)(\alpha^u+\alpha^v)}
\equiv\underline{\bx}_{p-1}^{(\mathbf{e}_i-\mathbf{g}_i)\alpha^w}
=\underline{\bx}_{p-1}^{\mathbf{m}_i^{(w)}} \mod K.
\end{split}
\end{align}
\item Let $k+1\leq i\leq n$.
We show that $x_{iu}x_{iv}-x_{iw}\in\left\langle\mathcal{G}\right\rangle$.
If $u,v,w\neq p-1$, then the following polynomial lies in $\left\langle\mathcal{G}\right\rangle$,
\begin{align*}
&\left( x_{iw}-x_{i,p-1}^{\alpha^w}\right)-x_{iu}\left(x_{iv}-x_{i,p-1}^{\alpha^v}\right)
-x_{i,p-1}^{\alpha^v}\left(x_{iu}-x_{i,p-1}^{\alpha^u}\right)\\
&=x_{iw}-x_{iu}x_{iv}-\left(x_{i,p-1}^{\alpha^w}-x_{i,p-1}^{\alpha^u}x_{i,p-1}^{\alpha^v}\right)\\
&\equiv x_{iu}-x_{iv}x_{iw}\mod K.
\end{align*}
If $v,w\neq p-1$ and $u=p-1$, then $\alpha^u+\alpha^v=1+\alpha^v=\alpha^w$ and
 the following polynomial is in $\left\langle\mathcal{G}\right\rangle$,
\begin{align*}
&\left( x_{iw}-x_{i,p-1}^{\alpha^w}\right)-x_{iu}\left(x_{iv}-x_{i,p-1}^{\alpha^v}\right)\\
&=x_{iw}-x_{iu}x_{iv}-\left(x_{i,p-1}^{\alpha^w}-x_{i,p-1}x_{i,p-1}^{\alpha^v}\right)\\
&\equiv x_{iw}-x_{iu}x_{iv}\mod K.
\end{align*}
The case $u,w\neq p-1$ and $v=p-1$ is analogous.

If $u,v\neq p-1$ and $w=p-1$, then $\alpha^u+\alpha^v=1$ and 
 the following polynomial is a member of $\left\langle\mathcal{G}\right\rangle$,
\begin{align*}
&x_{iu}\left(x_{iv}-x_{i,p-1}^{\alpha^v}\right)
+x_{i,p-1}^{\alpha^v}\left(x_{iu}-x_{i,p-1}^{\alpha^u}\right)\\
&=x_{iu}x_{iv}-x_{i,p-1}^{\alpha^u+\alpha^v}\\
&\equiv x_{iu}x_{iv}-x_{i,p-1}=x_{iu}x_{iv}-x_{iw}\mod K.
\end{align*}
If $u,v=p-1$ and $w\neq p-1$, then $\alpha^w=2$ and
\begin{align*}
 x_{iw}-x_{i,p-1}^{\alpha^w}=x_{iw}-x_{i,p-1}^2=x_{iw}-x_{iu}x_{iv}.
\end{align*}
The cases $u,w=p-1,v\neq p-1$ and $v,w=p-1, u\neq p-1$ cannot occur
since $\alpha^u,\alpha^v,\alpha^w\neq 0$. 
Similarly, the case $u,v,w=p-1$ is impossible.
\end{itemize}
\item 
Third we prove that the binomials in~$I_q$ whose second term is equal to~1 are generated by the binomials in $\mathcal{G}$. 
For this, let $\alpha^u+\alpha^v=0$ with $\alpha^u,\alpha^v\neq 0$ and consider the following cases:
\begin{itemize}
 \item Let $1\leq i\leq k$. Claim that $x_{iu}x_{iv}-1\in\left\langle\mathcal{G}\right\rangle$.
Indeed, calculating as in~(\ref{eq-auxMkuerzen}) gives
\begin{align*}
&x_{iv}\left(x_{iu}-\underline{\bx}_{p-1}^{\mathbf{m}_i^{(u)}}\right)
+\underline{\bx}_{p-1}^{\mathbf{m}_i^{(u)}}\left(x_{iv}-\underline{\bx}_{p-1}^{\mathbf{m}_i^{(v)}}\right)
\equiv x_{iu}x_{iv}-1\mod K.
\end{align*}
 \item Let $k+1\leq i\leq n$. Claim that $x_{iu}x_{iv}-1\in\left\langle\mathcal{G}\right\rangle$.
Indeed, if $u,v\neq p-1$, then 
\begin{align*}
&x_{iv}\left( x_{iu}-x_{i,p-1}^{\alpha^u}\right)-x_{i,p-1}^{\alpha^u}\left(x_{iv}-x_{i,p-1}^{\alpha^v}\right)
\equiv x_{iu}x_{iv}-1\mod K.
\end{align*}
If $u\neq p-1$ and $v=p-1$, then
\begin{align*}
&x_{iv}\left( x_{iu}-x_{i,p-1}^{\alpha^u}\right)\equiv x_{iu}x_{iv}-1\mod K.
\end{align*}
The case $u=p-1$ and $v\neq p-1$ is analogous and the case $u,v=p-1$ cannot occur.
\end{itemize}
\end{itemize}
\end{proof}

\begin{example}
Reconsider the ternary $[6,3]$ code $\Co$ given in Ex.~\ref{ex-1}.
For the associated generalized code ideal, we construct the reduced Gröbner basis w.r.t.\ the lex ordering according to Thm.~\ref{thm-lexGB}.
To this end, let $\F_3=\{0,\alpha=2,\alpha^2=1\}$ and so 
\begin{align*}
 \mathbf{m}_1^{(1)}=(0,0,0,2,2,0),\; \mathbf{m}_1^{(2)}=(0,0,0,1,1,0),\\
 \mathbf{m}_2^{(1)}=(0,0,0,1,1,0),\; \mathbf{m}_2^{(2)}=(0,0,0,2,2,0),\\
 \mathbf{m}_3^{(1)}=(0,0,0,1,2,1),\; \mathbf{m}_3^{(2)}=(0,0,0,2,1,2).
\end{align*}
Then the reduced Gröbner basis is given by
\begin{align*}
\mathcal{G}_\succ(I(\Co))&=\left\{x_{11}-x_{42}^2x_{52}^2,\, x_{12}-x_{42}^2x_{52}^2,\,
x_{21}-x_{42}x_{52},\, x_{22}-x_{42}^2x_{52}^2\right\}\\
&\cup\left\{x_{31}-x_{42}x_{52}^2x_{62},\, x_{32}-x_{42}^2x_{52}x_{62}^2,
x_{41}-x_{42}^2,\, x_{51}-x_{52}^2,\,x_{61}-x_{62}^2\right\}\\
&\cup\left\{ x_{42}^3-1,\,x_{52}^3-1,\,x_{62}^3-1\right\}. 
\end{align*}
\end{example}

\subsection{Gröbner bases for generalized code ideals}

The construction of the generalized code ideal $I_+(\Co)$ according to Thm.~\ref{thm-martShapeIc} can become quite tedious for larger problem instances. 
Indeed, for an $[n,k]$ code over $\F_q$ the generating set consists of $k(q-1)+n{q\choose 2}$ binomials.
In particular, if $q=p$ is a prime, Thm.~\ref{thm-lexGB} will provide an alternative generating set which is based on a standard generator matrix for the code and
consists of $k(p-1)+(n-k)(p-1)=n(p-1)$ binomials. 
This result can be extended to the general case. 
More precisely, the reduced Gröbner basis for the generalized code ideal $I_+(\Co)$ w.r.t.\ the lex ordering always consists of $n(q-1)$ binomials. 
 
To see this, let $q=p^r$, where $p$ is a prime and $r\geq 1$ is an integer.
The finite field $\F_q$ can be considered as a vector space over $\F_p$ with basis elements $1,\alpha,\alpha^2,\dots,\alpha^{r-1}$. 
Taking $\alpha$ as a primitive element of $\F_q$ and replacing $\alpha$ by $\alpha^{-1}$ gives 
another basis $\alpha^{q-r},\dots,\alpha^{q-2},\alpha^{q-1}$ for $\F_q$.
That is, each element $\beta$ of $\F_q$ can be uniquely written as
\begin{align}
 \beta=b_1\alpha^{q-r}+\dots+b_{r-1}\alpha^{q-2}+b_{r}\alpha^{q-1},\label{eq-basisFq}
\end{align}
where $b_1,\dots,b_r\in\F_p$.
Based on the expression~(\ref{eq-basisFq}), define the linear maps
\begin{align}
 \phi&:\F_q\rightarrow\F_p^{q-1}:\quad \beta\mapsto 
\left(0,\ldots,0,b_1,\dots,b_r\right),\label{eq-mapphi1}
\end{align}
and 
\begin{align}
 \phi^{(s)}&:\F_q^s\rightarrow\F_p^{(q-1)s}:\quad 
\left(\beta_1,\dots,\beta_s\right)\mapsto\left(\phi(\beta_1),\dots,\phi(\beta_s)\right).
\label{eq-mapphi2}
\end{align}
Note that $\phi\left(\alpha^u\right)=\mathbf{e}_u$ for each $q-r-1\leq u\leq q-1$.

\begin{lemma}\label{lem-blacktrianglephi}
 Let $\beta\in\F_q$.
If $\phi(\beta)$ is considered as a vector with integer entries, then
\begin{align*}
 \blacktriangledown\left(0,\dots,0,\phi(\beta),0,\dots,0\right)=(0,\dots,0,\beta,0,\dots,0).
\end{align*}
\end{lemma}
\begin{proof}
 Let $\beta=b_1\alpha^{q-r}+\dots+b_{r-1}\alpha^{q-2}+b_{r}\alpha^{q-1}$. 
Then $\phi(\beta)=(\mathbf{0},b_1,\dots,b_r)$ is mapped under $\blacktriangledown$ to
$0\cdot\alpha+\ldots+0\cdot\alpha^{q-r-1}+b_1\alpha^{q-r}+\ldots+b_r\alpha^{q-1}=\beta$. 
\end{proof}

\begin{theorem}\label{thm-GBcodeideal}
Let $\Co$ be an $[n,k]$ code over $\F_{q}$, where $q=p^r$, $p$ is prime and $r\geq 1$ is an integer.
Let $\alpha$ be a primitive element of\/ $\F_q$ and let $\mathbf{g}_1,\ldots,\mathbf{g}_k$
denote the rows of a standard generator matrix for $\Co$.
Let $\mathbf{m}_i^{(j)}$ be the projection of the vector $(\mathbf{e}_i-\mathbf{g}_i)\alpha^j$ onto the last
$n-k$ coordinates, $1\leq i\leq k$ and $1\leq j\leq q-1$. 
The reduced Gröbner basis for the generalized code ideal $I_+(\Co)$ w.r.t.\ the lex ordering 
$x_{11}\succ x_{12}\succ \ldots \succ x_{n,q-1}$
is given by
\begin{align}
\mathcal{G}
&=\left\{x_{ij}-\left(\bx_{k+1},\dots,\bx_{n}\right)^{\phi^{(n-k)}\left(\mathbf{m}_i^{(j)}\right)}
\mid  1\leq i\leq k,\,1\leq j\leq q\right\}\label{eq-GBgen1}\\
&\quad \cup \left\{x_{ij}-\bx_i^{\phi(\alpha^j)}\mid k+1\leq i\leq n,\,1\leq j\leq q-r-1\right\}\label{eq-GBgen2}\\
&\quad \cup \left\{x_{ij}^p-1\mid k+1\leq i\leq n,\,q-r\leq j\leq q-1\right\}.\label{eq-GBgen3}
\end{align}
\end{theorem}

\begin{proof}
The leading terms of all binomials in $\mathcal{G}$ are relatively prime and all the second terms do not involve any of the leading terms.
Hence, the set $\mathcal{G}$ is the reduced Gröbner basis for the ideal it generates.

It remains to prove that the binomials in $\mathcal{G}$ belong to the ideal $I_+(\Co)$ and that they form a generating set. 
First, claim that $\mathcal{G}\subset I_+(\Co)$.
Indeed, consider the following cases:
\begin{itemize}

 \item Take a binomial 
$x_{ij}-\left(\bx_{k+1},\dots,\bx_{n}\right)^{\phi^{(n-k)}\left(\mathbf{m}_i^{(j)}\right)}$ 
from the subset~(\ref{eq-GBgen1}). 
By Lemma~\ref{lem-blacktrianglephi}, applying the operator $\blacktriangledown$ to the exponent of the first and the second term
yields $\alpha^j\mathbf{e}_i$ and $\left(\mathbf{0},\mathbf{m}_i^{(j)}\right)$, respectively. 
But $\alpha^j\mathbf{e}_i-\left(\mathbf{0},\mathbf{m}_i^{(j)}\right)=\alpha^j\mathbf{g}_i\in\Co$
and so by Lemma~\ref{lem-reducemon},
the considered binomial belongs to $I_+(\Co)$.

 \item Pick a binomial $x_{ij}-\bx_i^{\phi(\alpha^j)}$ form the subset~(\ref{eq-GBgen2}). 
By Lemma~\ref{lem-blacktrianglephi}, applying the operator $\blacktriangledown$ to the first and the second term gives 
$\alpha^j\mathbf{e}_i$ and $\alpha^j\mathbf{e}_i$, respectively. 
Since $\mathbf{0}\in\Co$, this binomial lies in $I_+(\Co)$ by Lemma~\ref{lem-reducemon}.

\item Finally, consider a binomial $x_{ij}^p-1$ from the subset~(\ref{eq-GBgen3}).
By the same argument as in the previous case, the binomial $x_{ij}^p-1$ corresponds to the zero codeword and thus also belongs to $I_+(\Co)$.
\end{itemize}

Second, claim that $\mathcal{G}$ generates the ideal $I_+(\Co)$. 
Indeed, we show that the binomials in~(\ref{eq-codeidealGen}) are generated by $\mathcal{G}$.
For this, put 
$$J=\left\langle x_{ij}-\bx_i^{\phi(\alpha^j)}\mid k+1\leq i\leq n,\,1\leq j\leq q-r-1\right\rangle$$ and 
$$K=\left\langle x_{ij}^p-1\mid k+1\leq i\leq n,\,q-r\leq j\leq q-1\right\rangle.$$
Consider the following cases:
\begin{itemize}

 \item First we prove 
that the binomials in~$I_G$ are generated by the binomials in $\mathcal{G}$.  
For this, take a binomial $\bx^{\blacktriangle (\alpha^j\mathbf{g}_i)}-1$ 
for some $1\leq i\leq k$ and $1\leq j\leq p-1$.
Claim that
\begin{align*}
&\bx^{\blacktriangle(\alpha^j\mathbf{g}_i)}-1
\equiv\\
&\quad\quad \bx^{\blacktriangle\left(\mathbf{0},\,-\mathbf{m}_i^{(j)}\right)}
\left(x_{ij}-\left(\bx_{k+1},\dots,\bx_{n}\right)^{\phi^{(n-k)}\left(\mathbf{m}_i^{(j)}\right)}\right)        
\mod \left(J+K\right).
\end{align*}
Indeed, we have
$$\bx^{\blacktriangle\left(\mathbf{0},\,-\mathbf{m}_i^{(j)}\right)}=\left(\bx_{k+1},\dots,\bx_{n}\right)
^{\blacktriangle\left(-\mathbf{m}_i^{(j)}\right)}$$
and 
\begin{align*}
 \bx^{\blacktriangle\left(\mathbf{0},\,-\mathbf{m}_i^{(j)}\right)}x_{ij}
=\bx^{\blacktriangle\left(\mathbf{0},\,-\mathbf{m}_i^{(j)}\right)}\bx^{\blacktriangle\left(\alpha^j\mathbf{e}_i\right)}
= \bx^{\blacktriangle\left(\alpha^j\mathbf{g}_i\right)}.
\end{align*}
Furthermore,
\begin{eqnarray*}
\lefteqn{ \bx^{\blacktriangle\left(\mathbf{0},\,-\mathbf{m}_i^{(j)}\right)}
\left(\bx_{k+1},\dots,\bx_{n}\right)^{\phi^{(n-k)}\left(\mathbf{m}_i^{(j)}\right)}}\\
&& =\left(\bx_{k+1},\dots,\bx_{n}\right)
^{\blacktriangle\left(-\mathbf{m}_i^{(j)}\right)+\phi^{(n-k)}\left(\mathbf{m}_i^{(j)}\right)}.
\end{eqnarray*}
In order to show that this monomial reduces to $1$ modulo~$(K+J)$,
write $-\mathbf{m}_i^{(j)}=\left(\alpha^{j_1},\dots,\alpha^{j_{n-k}}\right)$.
Then we have
\begin{align*}
 \blacktriangle\left(-\mathbf{m}_i^{(j)}\right)&=\left(\mathbf{e}_{j_1},\dots,\mathbf{e}_{j_{n-k}}\right).
\end{align*}
We may assume that all entries in $\mathbf{m}_i^{(j)}$ are non-zero.
Moreover,
\begin{align*}
\phi^{(n-k)}\left(\mathbf{m}_i^{(j)}\right)
&=\left(\mathbf{0},b_{11},\dots,b_{1r},\dots,\mathbf{0},b_{n-k,1},\dots,b_{n-k,r}\right),
\end{align*}
where $\sum_{s=1}^{r}b_{\ell s}\alpha^{q-r-1+s}=-\alpha^{j_\ell}$ for $1\leq \ell\leq n-k$.
 
If $j_\ell\geq q-r$, then $b_{\ell s}=0$ for $s\neq j_\ell$ and $b_{\ell j_\ell}=p-1$. So 
the corresponding entry in the 
exponent vector $\blacktriangle\left(-\mathbf{m}_i^{(j)}\right)+\phi^{(n-k)}\left(\mathbf{m}_i^{(j)}\right)$
is $1+(p-1)=p$ giving rise to the monomial $x_{k+\ell, j_\ell}^p$. 
However, $x_{k+\ell, j_\ell}^p$ reduces to $1$ by the appropriate binomial in~$K$.

If $j_\ell\leq q-r-1$, then the monomial $x_{k+\ell,j_\ell}$ in $\left(\bx_{k+1},\dots,\bx_{n}\right)
^{\blacktriangle\left(-\mathbf{m}_i^{(j)}\right)}$ reduces to $\bx_\ell^{\phi\left(\alpha^{j_\ell}\right)}$
by the appropriate binomial in~$J$.
But $\phi\left(\alpha^{j_\ell}\right)=\left(\mathbf{0},p-b_{\ell 1},\dots,p-b_{\ell r}\right)$,
since $\phi$ is linear and
$\phi\left(-\alpha^{j_\ell}\right)=\left(\mathbf{0},b_{\ell_1},\dots,b_{\ell_r}\right)$.
This gives the monomials $x_{k+\ell, i}^{p}$ for $k+1\leq i\leq n$ which as above also reduce to $1$.

\item 
Second we prove that the binomials in $I_q$ whose second term is different from $1$
are generated by the binomials in $\mathcal{G}$.
To this end, let $\alpha^u+\alpha^v=\alpha^w$ with $\alpha^u,\alpha^v,\alpha^w\neq 0$ 
and consider the corresponding binomial $x_{iu}x_{iv}-x_{iw}$ for some $1\leq i\leq n$.
Claim that 
\begin{eqnarray*}
\lefteqn{x_{iv}\left(x_{iu}-\left(\bx_{k+1},\dots,\bx_{n}\right)^{\phi^{(n-k)}\left(\mathbf{m}_i^{(u)}\right)}\right)} \\
&&+(\bx_{k+1},\dots,\bx_{n})^{\phi^{(n-k)}\left(\mathbf{m}_i^{(u)}\right)}
\left(x_{iv}-\left(\bx_{k+1},\dots,\bx_{n}\right)^{\phi^{(n-k)}\left(\mathbf{m}_i^{(v)}\right)}\right)\\
&&-\left(x_{iw}-\left(\bx_{k+1},\dots,\bx_{n}\right)^{\phi^{(n-k)}\left(\mathbf{m}_i^{(w)}\right)}\right)\\
&&\equiv x_{iu}x_{iv}-x_{iw}\mod K.
\end{eqnarray*}

Indeed, 
$$\left(\bx_{k+1},\dots,\bx_{n}\right)^{\phi^{(n-k)}\left(\mathbf{m}_i^{(w)}\right)}
-\left(\bx_{k+1},\dots,\bx_{n}\right)^{\phi^{(n-k)}\left(\mathbf{m}_i^{(u)}\right)
+\phi^{(n-k)}\left(\mathbf{m}_i^{(v)}\right)}$$
reduces to zero, because 
$$\phi^{(n-k)}\left(\mathbf{m}_i^{(u)}\right)+\phi^{(n-k)}\left(\mathbf{m}_i^{(v)}\right)
=\phi^{(n-k)}\left(\mathbf{m}_i^{(w)}\right).$$

If $ k+1\leq i\leq n$, then we distinguish the following cases:
\begin{itemize}
 \item If $u,v,w\leq q-r-1$, then 
\begin{align*}
&x_{iv}\left(x_{iu}-\bx_i^{\phi(\alpha^u)}\right)
\!+\bx_i^{\phi(\alpha^u)}\left(x_{iv}-\bx_i^{\phi(\alpha^v)}\right)
\!-\!\left(x_{iw}-\bx_i^{\phi(\alpha^w)}\right) \\
&\equiv x_{iu}x_{iv}-x_{iw}\mod K
\end{align*}
since
 $\bx_i^{\phi\left(\alpha^u\right)}\cdot\bx_i^{\phi\left(\alpha^v\right)}
=\bx_i^{\phi\left(\alpha^u\right)+\phi\left(\alpha^v\right)}\equiv\bx_i^{\phi\left(\alpha^u+\alpha^v\right)}
=\bx_i^{\phi\left(\alpha^w\right)}\mod K.$

\item If $u,w\leq q-r-1$ and $q-r\leq v\leq q-1$, then 
\begin{align*}
x_{iv}\left(x_{iu}-\bx_i^{\phi(\alpha^u)}\right)
\!-\!\left(x_{iw}-\bx_i^{\phi(\alpha^w)}\right)\equiv x_{iu}x_{iv}-x_{iw}\mod K
\end{align*}
because $\phi\left(\alpha^v\right)=\mathbf{e}_v$ and so
\begin{eqnarray*}
 \lefteqn{\bx_i^{\phi(\alpha^w)}-x_{iv}\bx_i^{\phi(\alpha^u)}}\\
&&=\bx_i^{\phi(\alpha^w)}-\bx_i^{\phi(\alpha^v)}\bx_i^{\phi(\alpha^u)}
\equiv\bx_i^{\phi(\alpha^w)}-\bx_i^{\phi(\alpha^{v}+\alpha^u)}=0\mod K.
\end{eqnarray*}
The case $v,w\leq q-r-1$ and $q-r\leq u\leq q-1$ is analogous.

\item If $q-r\leq u,v\leq q-1$, then $w\leq q-r-1$ as the set 
$$\left\{\alpha^i\mid q-r-1\leq i\leq q-1\right\}$$ 
is a basis for $\F_q$.
But then $\phi\left(\alpha^w\right)=\mathbf{e}_u+\mathbf{e}_v$ and so
\begin{align*}
 x_{iw}-\bx_i^{\phi\left(\alpha^w\right)}=x_{iw}-x_{iu}x_{iv}.
\end{align*}

\item If $q-r\leq u,w\leq q-1$, then the same reasoning as in the preceding case shows that $v\leq q-r-1$.
But then $\alpha^v=\alpha^w+(p-1)\alpha^u$ and 
so $\phi\left(\alpha^v\right)=\mathbf{e}_w+(p-1)\mathbf{e}_u$.
Thus
\begin{align*}
 x_{iu}\left(x_{iv}-\bx_i^{\phi\left(\alpha^v\right)}\right)= 
x_{iu}\left(x_{iv}-x_{iw}x_{iv}^{p-1}\right)\equiv x_{iu}x_{iv}-x_{iw}\mod K.
\end{align*}
The case $q-r\leq v,w\leq q-1$ is similar and the case $q-r\leq u,v,w\leq q-1$ cannot occur.
\end{itemize}

\item Third we prove that the binomials in $I_q$ whose second term is equal to $1$ are generated by the
binomials in $\mathcal{G}$.
For this, let $\alpha^u+\alpha^v=0$ with $\alpha^u,\alpha^v\neq 0$ and take the binomial $x_{iu}x_{iv}-1$ for some
$1\leq i\leq n$.

If $1\leq i\leq k$, then
\begin{eqnarray*}
\lefteqn{ x_{iv}\left(x_{iu}-\left(\bx_{k+1},\dots,\bx_{n}\right)^{\phi^{(n-k)}\left(\mathbf{m}_i^{(u)}\right)}\right)}\\
&&+\left(\bx_{k+1},\dots,\bx_{n}\right)^{\phi^{(n-k)}\left(\mathbf{m}_i^{(u)}\right)}
\left(x_{iv}-\left(\bx_{k+1},\dots,\bx_{n}\right)^{\phi^{(n-k)}\left(\mathbf{m}_i^{(v)}\right)}\right)\\
&&\equiv x_{iu}x_{iv}-1\mod K
\end{eqnarray*}
since
\begin{eqnarray*}
\lefteqn{ \left(\bx_{k+1},\dots,\bx_{n}\right)
^{\phi^{(n-k)}\left(\mathbf{m}_i^{(u)}\right)+\phi^{(n-k)}\left(\mathbf{m}_i^{(u)}\right)}}\\
&&\equiv \left(\bx_{k+1},\dots,\bx_{n}\right)^{\phi^{(n-k)}\left(\mathbf{m}_i^{(u)}+\mathbf{m}_i^{(v)}\right)}\\
&&=\left(\bx_{k+1},\dots,\bx_{n}\right)^{\phi^{(n-k)}\left(\mathbf{0}\right)}=1\mod K.
\end{eqnarray*}

If $k+1\leq i\leq n$, then we distinguish the following cases:
\begin{itemize}
 \item If $u,v\leq q-r-1$, then
\begin{align*}
x_{iv}\left(x_{iu}-\bx_i^{\phi\left(\alpha^u\right)}\right)
+\bx_i^{\phi\left(\alpha^u\right)}\left(x_{iv}-\bx_i^{\phi\left(\alpha^v\right)}\right)
\equiv x_{iu}x_{iv}-1\mod K.
\end{align*}
\item If $u\leq q-r-1$ and $q-r\leq v\leq q-1$, 
then $\phi\left(\alpha^u\right)=(p-1)\mathbf{e}_v$
and so
\begin{align*}
x_{iv}\left(x_{iu}-\bx_i^{\phi\left(\alpha^u\right)}\right)=x_{iv}\left(x_{iu}-x_{iv}^{p-1}\right)
\equiv x_{iu}x_{iv}-1\mod K.
\end{align*}
The case $v\leq q-r-1$ and $q-r\leq u\leq q-1$ is analogous and the case $q-r\leq u,v\leq q-1$ cannot arise 
since $\alpha^u+\alpha^v=0$ leads to a contradiction.
\end{itemize}
\end{itemize}
\end{proof}

\begin{corollary}\label{cor-vdim}
 Let $\Co$ be an $[n,k]$ code over a finite field $\F_q$, where $q=p^r$.  
The generalized code ideal $I_+(\Co)$ is a zero-dimensional ideal and the dimension of the coordinate
ring $\KK[\bx]/I_+(\Co)$ as a $\KK$-vector space is $p^{r(n-k)}$.
\end{corollary}
\begin{proof}
By~\cite[p.39]{cls-app}, the algebra $\KK[\bx]/I_+(\Co)$ is finite-dimensional, 
since in the Gröbner basis $\mathcal{G}$ for $I_+(\Co)$ for each variable $x_{ij}$ there is a number
$m_{ij}\geq 0$ such that $x_{ij}^{m_{ij}}=\lt(g)$ for some $g\in\mathcal{G}$.
Moreover, its dimension is given by product of the degrees of the leading terms in the 
Gröbner basis $\mathcal{G}$ for $I_+(\Co)$.
\end{proof}

\begin{example}
Consider the finite field  $\F_9=\F_3[x]/(x^2+x+2)$ and take as primitive element $\alpha$ 
the root of the primitive polynomial $x^2+x+2$. Then the elements of $\F_9$ can be described as 
\begin{align*}
 &0, \alpha, \alpha^2=2\alpha\!+\!1, \alpha^3=2\alpha\!+\!2, \alpha^4=2, \alpha^5=2\alpha, 
\alpha^6=\alpha\!+\!2, \alpha^7=\alpha\!+\!1, \alpha^8=1.
\end{align*} 
 Consider the $[3,2]$ code $\Co$ over $\F_9$ with generator matrix  
\begin{align*}
 G=\begin{pmatrix}1&0&\alpha^2\\0&1&\alpha^5
   \end{pmatrix}=\begin{pmatrix}1&0&-\alpha^6\\0&1&-\alpha
   \end{pmatrix}.
\end{align*}
Table~\ref{tab-mapphi} shows the images of all non-zero elements of $\F_9$ and 
of all $\mathbf{m}_i^{(j)}$ under the map $\phi$. 
This table can be used to construct the reduced Gröbner basis $\mathcal{G}$ 
for $I_+(\Co)$ w.r.t.\ lex ordering according to Thm.~\ref{thm-GBcodeideal}:
\begin{align*}
\mathcal{G}(I_+(\Co))=&\{ x_{11}-x_{37},\,x_{12}-x_{38},\,x_{13}-x_{37}x_{38}^2,\,x_{14}-x_{37}^2x_{38}^2,\\
&\quad\quad x_{15}-x_{37}^2,\,x_{16}-x_{38}^2,\,x_{17}-x_{37}^2x_{38},\,x_{18}-x_{37}x_{38},\,\}\\
&\cup\{x_{21}-x_{37}^2x_{38}^2,\,x_{22}-x_{37}^2,\,x_{23}-x_{38}^2,\,x_{24}-x_{37}^2x_{38},\\
&\quad\quad x_{25}-x_{37}x_{38},\,x_{26}-x_{37},\,x_{27}-x_{38},\,x_{28}-x_{37}x_{38}^2,\,\}\\
&\cup\{x_{31}-x_{37}x_{38}^2,\,x_{32}-x_{37}^2x_{38}^2,\,x_{33}-x_{37}^2,\,x_{34}-x_{38}^2,\\
&\quad\quad x_{35}-x_{37}^2x_{38},\,x_{36}-x_{37}x_{38}\}\\
&\cup\{x_{37}^3-1,\,x_{38}^3-1\}.
\end{align*}
\end{example}
\begin{table}[htb]
 \begin{tabular}{c|c|c | c | c | c}
  $\beta\in\F_9$&$\phi(\beta)\in\F_3^8$&$\mathbf{m}_1^{(j)}$&$\phi\left(\mathbf{m}_1^{(j)}\right)$
&$\mathbf{m}_2^{(j)}$&$\phi\left(\mathbf{m}_2^{(j)}\right)$\\\hline
	  $\alpha$&$(\mathbf{0},1,2)$&$\mathbf{m}_1^{(1)}=\alpha^7$&$(\mathbf{0},1,0)$&$\mathbf{m}_2^{(1)}=\alpha^2$&$(\mathbf{0},2,2)$\\
	  $\alpha^2$&$(\mathbf{0},2,2)$&$\mathbf{m}_1^{(2)}=\alpha^8$&$(\mathbf{0},0,1)$&$\mathbf{m}_2^{(2)}=\alpha^3$&$(\mathbf{0},2,0)$\\
	  $\alpha^3$&$(\mathbf{0},2,0)$&$\mathbf{m}_1^{(3)}=\alpha$&$(\mathbf{0},1,2)$& $\mathbf{m}_2^{(3)}=\alpha^4$&$(\mathbf{0},0,2)$\\
	  $\alpha^4$&$(\mathbf{0},0,2)$&$\mathbf{m}_1^{(4)}=\alpha^2$&$(\mathbf{0},2,2)$&$\mathbf{m}_2^{(4)}=\alpha^5$&$(\mathbf{0},2,1)$\\
	  $\alpha^5$&$(\mathbf{0},2,1)$&$\mathbf{m}_1^{(5)}=\alpha^3$&$(\mathbf{0},2,0)$&$\mathbf{m}_2^{(5)}=\alpha^6$&$(\mathbf{0},1,1)$\\
	  $\alpha^6$&$(\mathbf{0},1,1)$&$\mathbf{m}_1^{(6)}=\alpha^4$&$(\mathbf{0},0,2)$&$\mathbf{m}_2^{(6)}=\alpha^7$&$(\mathbf{0},1,0)$\\
	  $\alpha^7$&$(\mathbf{0},1,0)$&$\mathbf{m}_1^{(7)}=\alpha^5$&$(\mathbf{0},2,1)$&$\mathbf{m}_2^{(7)}=\alpha^8$&$(\mathbf{0},0,1)$\\
	  $\alpha^8$&$(\mathbf{0},0,1)$&$\mathbf{m}_1^{(8)}=\alpha^6$&$(\mathbf{0},1,1)$&$\mathbf{m}_2^{(8)}=\alpha$&$(\mathbf{0},1,2)$\\
\end{tabular}
\caption{Evaluation of the map $\phi$}
\label{tab-mapphi}
\end{table}

Besides the already mentioned benefit that the generating set given by Thm.~\ref{thm-GBcodeideal} consists
of less binomials than the generating set provided by Thm.~\ref{thm-martShapeIc}, there is the computational
advantage that the generating set in Thm.~\ref{thm-GBcodeideal} is already a reduced Gröbner basis. 
Once the reduced Gröbner basis for $I_+(\Co)$ w.r.t.\ the lex ordering has been constructed,
the FGLM algorithm can efficiently compute the reduced Gröbner basis w.r.t.\ any other monomial order~\cite{fglm}.
The complexity of the FGLM algorithm is $O(mD^3)$, where $m$ denotes the number of variables and $D$ the dimension of
$\KK$-algebra $\KK[\bx]/I$, provided that there is no growth of the coefficients. 
But unfortunately, this does not yield an improvement over the ordinary Gröbner basis computation by Buchberger's method
due to the large dimension of the algebra $\KK[\bx]/I_+(\Co)$. 
Since the ideal $I_+(\Co)$ is zero-dimensional and there is no growth of the coefficients 
(this can also be enforced by taking $\KK=\F_2$) 
the runtime amounts to $O((n(q-1)p^{3r(n-k)})$.
Unfortunately, no considerable improvement is achieved by this approach when compared with the FGML method 
proposed in~\cite{martcodeideal, marquez_corbella} whose complexity is
$O(n^2(q-1)p^{r(n-k)})$~(\cite[Thm.~4.42]{marquez_corbella}).


\section{Heuristic Decoding of Linear Codes}

It has been shown that in the binary case the code ideal $I(\Co)$ can be exploited to establish a 
general method for complete decoding~\cite{borges}. 
The key ingredients are the reduced Gröbner basis with respect to a degree compatible ordering and the division algorithm. 
The proposed algorithm is based on the fact that for any vector $w\in\F_2^n$, its Hamming weight 
coincides with the degree of the monomial $\bx^w$, i.e., $\wt(w)=\deg(\bx^w)$. 
However, in the non-binary case this is generally not true. 

The idea behind the introduction of the generalized code ideal in~\cite{martcodeideal}
is that it overcomes this deficiency. Indeed, the generalized code ideal solves the complete decoding
problem and provides a method that for any received word returns the closest codeword w.r.t.\ the Hamming 
distance~\cite[Chapter 4]{marquez_corbella}. 
The crucial result is the following.

\begin{theorem}{\cite[Thm.~4.19]{marquez_corbella}}\label{thm-completedecode}
Let $\Co$ be a linear code over a finite field with error-correcting capability $t$ and let $\mathcal{G}$
be the reduced Gröbner basis for $I_+(\Co)$ w.r.t.\ a degree compatible ordering. 
For each word $w\in\F_q^n$, the remainder of $\bx^w$ on division by $\mathcal{G}$ yields a monomial $\bx^e$
such that $e$ is a closest codeword to~$w$ w.r.t.\ the Hamming distance. 
If additionally, $\supp(e)\leq t$, then $w-e$ is the unique closest codeword to $w$.
\end{theorem}

This result implies that the standard monomials of $I_+(\Co)$ w.r.t.\ a degree compatible ordering 
provide a set of coset leaders.
Unfortunately, the decoding according to the above result can be rather costly due to the large number of variables and the high complexity of 
Buchberger's algorithm. 
Although the degree reverse lexicographic order is usually the best choice from a complexity point of view, 
the computations can still become infeasable. 
Indeed, for a zero-dimensional ideal the computational complexity $O(d^{m^2})$ is exponential in the number $m$ of variables, where 
$d$ is the maximal degree of the defining polynomials~\cite{CanigliaGH88,fglm}. 
And indeed, this is hardly surprising since the
complete decoding problem is known to be NP-hard~\cite{berlekamp}.

In the following, we present a heuristic method for decoding based on the code ideal $I(\Co)$.
To this end, let $\Co$ be an $[n,k,d]$ code over $\F_p$, where $p$ is a prime, 
and let $\succ$ be any degree compatible order on $\KK[x_1,\ldots,x_n]$.
\begin{proposition}\label{prop-transversal}
 The standard monomials of the $\KK$-algebra $\KK[\bx]/I(\Co)$ with respect to any monomial order give rise to
a transversal of the elements in the quotient space $\F_p^n/\Co$.
\end{proposition}
\begin{proof}
A word $w\in\F_p^n$ is a codeword if and only if the remainder of $\bx^w$ on division by a Gröbner basis
for $I(\Co)$ equals~$1$. 
So if $\bx^w$ is a standard monomial with $w\neq\mathbf{0}$, then $w\notin\Co$. 
Claim that the standard monomials correspond one-to-one to the cosets of $\Co$. 
Indeed, assume that two distinct standard monomials $\bx^u$ and $\bx^v$ correspond to the same coset $u+\Co=v+\Co$. 
Then $u-v\in\Co$ and so $\bx^u-\bx^v\in\Co$. 
But then either of them must be the leading term which is a contradiction to both being standard monomials.
Furthermore, the dimension of the algebra $\KK[\bx]/I_+(\Co)$ equals the dimension of the algebra $\KK[\bx]/\lt(I_+(\Co))$,
which is given by the number of standard monomials.
By Cor.~\ref{cor-vdim}, the number of standard monomials equals the number of cosets of $\Co$.
\end{proof}

Note that this result is also true for the generalized code ideal and 
that part of the Gröbner representation in~\cite{bqbtmcmmCoset} consists of exactly such a transversal of $\F_p^n/\Co$.
Indeed, the complete decoding algorithm proposed in~\cite{borges, marquez_corbella} is based on the fact that 
for a degree compatible ordering the transversal as in Prop.~\ref{prop-transversal} gives 
the set of coset representatives of minimal Hamming weight.
For the code ideal $I(\Co)$, however, this is in general not true. 

\begin{proposition}
Let $w\in\F_p^n$ be a word. 
If the remainder of $\bx^w$ on division by $\mathcal{G}_\succ(I(\Co))$
is a monomial $\bx^e$ such that the Hamming weight of $e$ is less than or equal to the error-correcting capability, 
then $w-e$ is the unique closest codeword to $w$ w.r.t.\ the Hamming metric.
\end{proposition}
\begin{proof}
 Let $c$ be the transmitted codeword and $w$ the received word which contains at most 
$t=\left\lfloor{\frac{d-1}{2}}\right\rfloor$ errors, i.e., $w=c+e$ for some $e\in\F_p^n$ 
with $\wt(e)\leq t$. Let $\bx^f$ be the remainder of 
$\bx^w$ on division by $\mathcal{G}$ with $\wt(f)\leq t$. 
Then $\bx^w-\bx^f\in I(\Co)$ and we see that $w-f$ is a codeword. 
But there is exactly one codeword in the ball around $w$ with radius $t$ and so   
$\dist(w,w-f)\leq t$ implies that $w-f=c$ and hence $e=f$.
\end{proof}

\begin{example}\label{ex-decode1} 
Consider the $[7,2,5]$ code $\Co$ over $\F_3 =\{0,1,2\}$ with generator matrix
\begin{align*}
 G=\begin{pmatrix}
    1&0&1&2&1&1&1\\0&1&2&2&1&0&2
   \end{pmatrix}.
\end{align*}
The associated code ideal is
$$I(\Co)=\left\langle x_1-x_3^2x_4x_5^2x_6^2x_7^2, x_2-x_3x_4x_5^2x_7\right\rangle
+\left\langle x_i^3-1\mid 3\leq i\leq 7\right\rangle.$$
Moreover, let $\mathcal{G}$ be the reduced Gröbner basis for $I(\Co)$ w.r.t.\ the 
degree reverse lexicographical ordering~$\succ$.  

Assume that the codeword $c=(1,2,2,0,0,1,2)$ has been transmitted and the word $w=(0,2,2,0,0,0,2)$ has been received. 
Since the code has error-correcting capability $t=2$, the errors in $w$ should be correctable.
This can be accomplished in the code ideal $I(\Co)$.
For this, the division of the remainder of $x_2^2x_3^2x_7^2$ by $\mathcal{G}$ yields the monomial 
$x_1^2x_6^2$ and thus the error vector $(2,0,0,0,0,2,0)$. 
Thus the decoding $(0,2,2,0,0,0,2)-(2,0,0,0,0,2,0)=(1,2,2,0,0,1,2)$ yields the codeword sent.

However, if the received word is $w'=(0,1,2,0,0,1,2)$, then division of $x_2x_3^2x_6x_7^2$ by $\mathcal{G}$ 
provides the remainder $x_4x_5^2x_6$ and thus the codeword $c'=(0,1,2,2,1,0,2)$ 
which is different from $c$ although $w'$ has only two 
erroneous positions~\cite[Examples 4.9 and 4.29]{marquez_corbella}.
\end{example}

The last example illustrates the shortcomings of the code ideal. 
When the word $w'=(0,1,2,0,0,1,2)$ is received,  the error vector $e'=(0,0,0,1,2,1,0)$ is obtained instead of $e=(2,2,0,0,0,0,0)$.
The reason is that $x_1^2x_2^2\succ x_4x_5^2x_6$ and thus $x_4x_5^2x_6$ is the appointed coset representative 
(and in the sense of Prop.~\ref{prop-transversal}). 
Many such examples can be constructed where the decoding by the ideal $I(\Co)$ does not work correctly (in the sense of nearest neighbor decoding).
Nonetheless, there are heuristics which may overcome this problem.
For instance, Alg.~\ref{alg-decodingHeuristic} provides such a heuristic making use of the following 
three subroutines:
\begin{itemize}
 \item $\texttt{modulo}(w,p)$ applied to $w\in\ZZ^n$ and a prime number $p$ returns the vector $w$ whose
	components are reduced modulo $p$.
 \item $\texttt{invert}(\alpha,p)$ applied to an integer $\alpha<p$ returns an integer $\alpha^{-1}$ 
	such that $\alpha\cdot \alpha^{-1}=1\mod p$.
 \item $\texttt{reduce}(\bx^w,\mathcal{G})$ applied to a monomial $\bx^w$ and a set of polynomials $\mathcal{G}$
	 returns the remainder of $\bx^w$ on division by $\mathcal{G}$.
\end{itemize}

\renewcommand{\algorithmicrequire}{\textbf{Input:}}
\renewcommand{\algorithmicensure}{\textbf{Output:}}

\begin{algorithm}[htb!]
\caption{Heuristic for nearest neighbor decoding}
\begin{algorithmic}[1] 
    \REQUIRE Gröbner basis $\mathcal{G}(I(\Co))$ w.r.t.\ a degree compatible order, 
	      error correcting capability $t$, prime number $p$, 
	      and received word $w$ 
    \ENSURE Either a codeword $c$ such that $\dist(c,w)\leq t$ or \textit{fail}
    \STATE $i=1;\, c:=\mathbf{0};$
    \WHILE{$i< p$}
	\STATE $w=\texttt{modulo}(iw,p);$
        \STATE $\bx^e=\texttt{reduce}( \bx^w,\mathcal{G});$
	\IF{$\supp(e)\leq t$}
	\STATE $c=\texttt{modulo}(w-e,p);$
	\STATE $i^{-1}=\texttt{invert}(i,p);$
	\STATE $c=\texttt{modulo}(i^{-1}c,p);$
        \RETURN $c$
        \ELSE  
        \STATE $i=i+1;$
	\ENDIF
    \ENDWHILE
    \RETURN $c=$ \textit{fail}
\end{algorithmic}
\label{alg-decodingHeuristic}
\end{algorithm}

\begin{example}
Revisit Ex.~\ref{ex-decode1} and apply Alg.~\ref{alg-decodingHeuristic} with the parameters
$t=2$, $p=3$ and $w=w'=(0,1,2,0,0,1,2)$. For $\alpha=1$ the remainder $x_4x_5^2x_6$ of $x_2x_3^2x_6x_7^2$ 
is computed whose
exponent clearly does not satisfy the condition in line 5. For $\alpha=2$, however, the remainder 
of $x_2^2x_3x_6^2x_7$ is $x_1x_2$ and thus Alg.~\ref{alg-decodingHeuristic} yields the correct
codeword $c=(1,2,2,0,0,1,2)$.
\end{example}

The advantage of this heuristic is that if it does give the nearest codeword to the received word,
then we can recognize this (assuming that the minimum distance is known). 
However, if the output is \textit{fail}, 
then chances are that the received word contains at most $t$ errors. 

In some situations it can be guaranteed that the above heuristic gives a result other than \textit{fail}. 
For this, note that $\alpha e+\Co=\alpha\left(e+\Co\right)$ for all $\alpha\in\F_p^\ast$. 
This means that $\alpha w$ lies in the coset $\alpha e+\Co$ if and only if $w$ lies in the coset $e+\Co$.
If the representative of minimal Hamming in the coset $w+\Co$ does not coincide with the
representative appointed by the reduced Gröbner basis w.r.t.\ the chosen degree compatible monomial order, 
Alg.~\ref{alg-decodingHeuristic} jumps to the cosets $2(w+\Co)$, $3(w+\Co),\dots$ and so on. 
Some situations in which the algorithm will work are as follows:
\begin{itemize}
 \item For codes with minimum distance $d=3$ or $d=4$ (i.e., when $1$ error can be corrected) all error
patterns not exceeding the error correcting capability can be corrected by Alg.~\ref{alg-decodingHeuristic}.
\item For all codes errors not exceeding the error correcting capability can be corrected by
Alg.~\ref{alg-decodingHeuristic} whose error values are all the same.
\end{itemize}

Moreover, for a fixed code it can be tested for which error patterns 
Alg.~\ref{alg-decodingHeuristic} will fail.
Consider a word $w=c+e$ with $\wt(e)\leq t$ such that
the remainder of $\bx^w$ is $\bx^f$ with $\wt(f)>t$. Clearly, $\bx^f$ is a standard monomial and
$w-e\in\Co$ and $w-f\in\Co$ imply that $e-f\in\Co$ and so $\bx^e-\bx^f\in I(\Co)$. 
Since $\bx^f$ is a standard monomial it follows that the monomial $\bx^e$ must belong to $\lt(I(\Co))$.
Hence, the algorithm will fail if $\bx^{\alpha e}\in\lt(I(\Co))$ for all $\alpha\in\F_p^\ast$.
Consequently, the decoding algorithm~\ref{alg-decodingHeuristic} can be incomplete but when it works it provides a reduction in complexity.



\bibliographystyle{plain}
\bibliography{refBook,refPaper}

\end{document}

%% file: somedefs.tex
\newcommand{\Co}{\mathcal{C}}

\newcommand{\ZZ}{\mathbb{Z}}
\newcommand{\KK}{\mathbb{K}}
\newcommand{\NN}{\mathbb{N}}
\newcommand{\F}{\mathbb{F}}

\newcommand{\bx}{\mathbf{x}}
\newcommand{\by}{\mathbf{y}}

\newcommand{\lt}{\text{lt}}

\newcommand{\supp}{\text{supp}}
\newcommand{\wt}{\text{wt}}
\newcommand{\dist}{\text{dist}}